\UseRawInputEncoding\documentclass[12pt, reqno]{amsart}
\usepackage[margin=1in]{geometry}
\usepackage{amssymb,latexsym,amsmath,amscd,amsfonts}
\usepackage{latexsym}
\usepackage[mathscr]{eucal}
\usepackage{bm}
\usepackage{tabularx}
\usepackage{mathptmx}
\usepackage{amssymb}
\usepackage{amsthm}
\usepackage{dcolumn}
\usepackage[all]{xy}
\usepackage{enumitem}
\usepackage[utf8]{inputenc}

\def \qed {\hfill \vrule height6pt width 6pt depth 0pt}
\def\textmatrix#1&#2\\#3&#4\\{\bigl({#1 \atop #3}\ {#2 \atop #4}\bigr)}
\def\dispmatrix#1&#2\\#3&#4\\{\left({#1 \atop #3}\ {#2 \atop #4}\right)}
\newcommand{\beg}{\begin{equation}}
	\newcommand{\eeg}{\end{equation}}
\newcommand{\ben}{\begin{eqnarray*}}
	\newcommand{\een}{\end{eqnarray*}}

\newcommand{\C}{\mathbb C}
\newcommand{\R}{\mathbb R}
\newcommand{\N}{\mathbb N}
\newcommand{\T}{\mathbb T}

\newcommand{\pr}{\partial}
\newcommand{\ol}{\overline}

\newcommand{\LL}{\mathbb L}
\newcommand{\F}{\mathbb F}
\newcommand{\E}{\mathbb E}
\newcommand{\He}{\mathbb H}
\newcommand{\Pe}{\mathbb P}
\newcommand{\D}{\mathbb D}
\newcommand{\G}{\mathbb G}
\newcommand{\QN}{\mathbb H_N}
\newcommand{\Gg}{\mathbb G_2}
\newcommand{\Ebar}{\overline{\mathbb E}}
\newcommand{\Pbar}{\overline{\mathbb P}}

\newcommand{\lm}{\lambda}
\newcommand{\diag}{\text{diag}}

\newcommand{\hexa}{\text{hexa}}

\newtheorem{thm}{Theorem}[section]
\newtheorem{cor}[thm]{Corollary}
\newtheorem{lem}[thm]{Lemma}

\newtheorem{prop}[thm]{Proposition}
\numberwithin{equation}{section} \theoremstyle{definition}

\newtheorem{eg}[thm]{Example}

\begin{document}
	\title{A domain in $\mathbb C^4$ and its connection with $\mu$-synthesis problem}

\author{SOURAV PAL AND NITIN TOMAR}

\address[Sourav Pal]{Mathematics Department, Indian Institute of Technology Bombay,
	Powai, Mumbai - 400076, India.} \email{sourav@math.iitb.ac.in}

\address[Nitin Tomar]{Mathematics Department, Indian Institute of Technology Bombay, Powai, Mumbai-400076, India.} \email{tomarnitin414@gmail.com}

\keywords{The domain $\mathbb{F}$, hexablock, $\mu$-synthesis problem, distinguished boundary}	

\subjclass[2020]{47A13, 47A20, 47A25, 47A45}	

	\begin{abstract}

		The hexablock $\mathbb{H}$, introduced by Biswas-Pal-Tomar \cite{Hexablock}, is a Hartogs domain in $\mathbb{C}^4$ fibered over the tetrablock $\mathbb{E}$ in $\mathbb{C}^3$ defined as
			\[
		\mathbb H= \left\{(a, x_1, x_2, x_3) \in \mathbb{C} \times \mathbb{E} :  \sup_{z_1, z_2 \in \mathbb{D}}\left|\frac{a\sqrt{(1-|z_1|^2)(1-|z_2|^2)}}{1-x_1z_1-x_2z_2+x_3z_1z_2}\right|<1 \right\},
		\]
where $\mathbb{D}=\{z \in \mathbb{C}: |z|<1\}$. The domain hexablock arises in the context of $\mu$-synthesis problem. While the symmetrized bidisc and the tetrablock are well-known domains associated with the $\mu$-synthesis problem, they can also be realized as $2$-proper holomorphic images of classical Cartan domains. Motivated by this, Ghosh-Zwonek \cite{GhoshI} introduced a family of domains $\{\mathbb{L}_n: n \geq 2\}$ and study its function theoretic and geometric properties. It was further proved that $\mathbb{L}_2$ and $\mathbb{L}_3$ are biholomorphic to the symmetrized bidisc $\mathbb{G}_2$ and the tetrablock $\mathbb{E}$, respectively. Moreover, the domain $\mathbb{L}_4$ is biholomorphic to the domain given by
		\[
		\mathbb{F}=\left\{(a_{11}, a_{22}, \det(A), a_{21}+a_{12}) \in \mathbb{C}^4 : A=(a_{ij})_{i,j=1}^2 \in M_2(\mathbb{C}), \|A\| < 1 \right\}.
		\]
Both $\mathbb{F}$ and $\mathbb{H}$ are domains in $\mathbb{C}^4$ closely related to the tetrablock $\mathbb{E}$ in the sense that if $(x, a, p, s) \in \mathbb{F}$, then $(x, a, p) \in \mathbb{E}$, and if $(a, x_1, x_2, x_3) \in \mathbb{H}$, then $(x_1, x_2, x_3) \in \mathbb{E}$. Moreover, $(x, a, p, 0) \in \F$ if and only if $(0, x, a, p) \in \He$. Also, $(a_{21}, a_{11}, a_{22}, \det(A)) \in \mathbb{H}$ for $A=(a_{ij})_{i,j=1}^2 \in M_2(\mathbb{C})$ with $\|A\|<1$. This structural similarities between $\mathbb{F}$ and $\mathbb{H}$ raise the natural question if these two domains are biholomorphic. In this paper, we answer this question in negative. Moreover, we provide alternative characterizations of the domain $\mathbb{F}$ and explore its connection with the domains associated with $\mu$-synthesis problem such as the symmetrized bidisc, the tetrablock, the pentablock and the hexablock. We also address the following question posed by Ghosh-Zwonek in \cite{GhoshI}: does $\mathbb{F}$ arise from a $\mu$-synthesis problem in the same manner as $\mathbb{G}_2$ and $\mathbb{E}$ ?
		
	\end{abstract}
	\maketitle
	
		\section{Introduction}\label{sec_01}
	
	\noindent Throughout the paper, the sets of complex numbers, real numbers, rational numbers, integers and natural numbers are denoted by $\mathbb{C}, \mathbb{R}, \mathbb{Q}, \mathbb{Z}$ and $\mathbb{N}$, respectively. The symbols $\mathbb{D}, \mathbb{T}$ stand for the open unit disc and the unit circle in the complex plane with their centres at the origin, respectively. We write $M_n(\mathbb{C})$ for the space of all $n \times n$ complex matrices. The transpose, adjoint, trace, determinant, spectral radius and operator norm of a matrix $A \in M_n(\mathbb{C})$ are denoted by $A^t, A^*, \text{tr}(A)$, $\det(A)$, $r(A)$ and $\|A\|$, respectively.	
	
	\smallskip
	
In control engineering (see \cite{Doyle, DoyleII}), the structured singular value of an $n \times n$ matrix $A$ is a cost function that generalizes the classical operator norm of $A$, incorporating structural information about the perturbations under consideration. It is computed relative to a specified linear subspace $E$ of $M_n(\mathbb{C})$, which is commonly referred to as the structure. Given such a subspace $E \subseteq M_n(\mathbb{C})$, the structured singular value is defined by
	\begin{equation} \label{eqn:NEW-01}
		\mu_E(A)=\frac{1}{\inf\{\|X\| \ : \ X \in E, \ \det(I-AX)=0 \}}   \quad (A \in M_n(\mathbb{C}))
	\end{equation} 
	with the understanding that $\mu_E(A)=0$ if there does not exist $X \in E$ satisfying $\det(I-AX)=0$. Efforts to solve different cases of the $\mu$-synthesis problem give rise to various domains in $\C^2$ and $\C^3$ such as symmetrized bidisc \cite{AglerIII}, tetrablock \cite{Abouhajar} and pentablock \cite{AglerIV}. These domains are denoted by $\G_2, \mathbb{E}$ and $\mathbb{P}$ respectively and are defined as follows:
	\begin{align*}
		& \mathbb{G}_2 
		=\{(z_1+z_2, z_1z_2) \in \C^2 : |z_1| <1, |z_2|<1\}, \\
		& \E   =\{(x_1, x_2, x_3) \in \C^3 \ : \ 1-x_1z_1-x_2z_2+x_3z_1z_2 \ne 0 \ \ \text{for} \ \ |z_1| \leq 1, |z_2| \leq 1 \}, \\
		& \mathbb{P}  =\{(a_{21}, \text{tr}(A), \det(A)) \in \C^3 : A=(a_{ij}) \in M_2(\C), \|A\| <1\}. 
	\end{align*}
	The domains $\G_2, \E$ and $\Pe$ have a rich literature on geometry and function theory independent of their connections with $\mu$-synthesis and control engineering. We refer to \cite{AglerYoung}-\cite{Agler2008} and \cite{Tirtha_Pal, Costara2004, Costara2005, NikolovIII, sourav14, pal-shalit, PZ} for a further reading on $\G_2$. Similarly, a reader can see \cite{Abouhajar, Alsalhi, Tirtha, Kosi, zwonek1, Young, Zwonek} and \cite{AglerIV, Alsheri, Jindal, KosinskiII, GuicongII, Guicong} respectively, to follow some interesting works on the tetrablock and the pentablock. Besides their rich geometry and function theory, these domains have also attracted attention in operator theory. In this direction, a reader is referred to the works \cite{Agler2003, AglerYoung2000, Tirtha_Pal, tirtha-sourav1} for the symmetrized bidisc, \cite{Tirtha, Ball_Sau, Pal_E3} for the tetrablock and \cite{JindalII, PalN} for the pentablock. Also, see the references therein.

\smallskip

Recently Biswas-Pal-Tomar  \cite{Hexablock} introduced another domain in $\C^4$ in connection with the $\mu$-synthesis problem. Let $E$ be the linear subspace of all $2\times 2$ upper triangular complex matrices and let us denote the corresponding cost function $\mu_E$ in this case by $\mu_{\text{hexa}}$. Consider the mapping
	\[
	\pi: M_2(\mathbb{C}) \to \mathbb{C}^4, \quad A = \begin{pmatrix} a_{11} & a_{12} \\ a_{21} & a_{22} \end{pmatrix} \mapsto (a_{21}, a_{11}, a_{22}, \det(A)).
	\]
	The set  $\mathbb{H}_\mu = \left\{ \pi(A) : A \in M_2(\mathbb{C}),  \mu_{\text{hexa}}(A) < 1 \right\}$ is refereed to as the $\mu$-hexablock in \cite{Hexablock}. However, $\mathbb{H}_\mu$ is connected but not open in $\mathbb{C}^4$, and therefore not a domain in $\mathbb{C}^4$. 	The authors of \cite{Hexablock} introduced a new family of fractional linear maps given by
	\begin{equation}\label{linear}
		\Psi_{z_1,z_2}(a, x_1, x_2, x_3) := \frac{a \sqrt{(1 - |z_1|^2)(1 - |z_2|^2)}}{1 - x_1 z_1 - x_2 z_2 + x_3 z_1 z_2}, \quad ((z_1, z_2) \in \ol{\D}^2, (x_1, x_2, x_3) \in \E).
	\end{equation}
		Note that $1 - x_1 z_1 - x_2 z_2 + x_3 z_1 z_2 \neq 0$ for any $z_1, z_2 \in \ol{\D}$ since $(x_1 x_2, x_3) \in \E$. Define  
	\begin{align}\label{eqn_He}
	\mathbb{H} := \left\{ (a, x_1, x_2, x_3) \in \mathbb{C} \times \mathbb{E} : \sup_{z_1,z_2 \in \mathbb{D}} \left| \Psi_{z_1,z_2}(a, x_1, x_2, x_3) \right| < 1 \right\}.
	\end{align}
The set $\mathbb{H}$, introduced by Biswas–Pal–Tomar \cite{Hexablock}, is called the hexablock due to the geometric structure of its real slice $\mathbb{H} \cap \mathbb{R}^4$, which is a bounded set whose boundary comprises exactly four extreme sets and two hypersurfaces.	The hexablock $\mathbb{H}$ is a domain in $\mathbb{C}^4$ and $\mathbb{H} = \operatorname{int}(\overline{\mathbb{H}}_\mu)$. Its connections with the symmetrized bidisc $\G_2$, the tetrablock $\E$ and the pentablock $\Pe$ have been explored in \cite{Hexablock}. In fact, it was proved in \cite{Hexablock} that all these three domains are holomorphic retracts of the hexablock. Furthermore, $\mathbb{G}_2$ and $\E$ can be realized as the images of $2$-proper holomorphic images of classical Cartan domains of type $IV$ in $\C^2$ and $\C^3$ respectively. To this end, Ghosh-Zwonek \cite{GhoshI} introduced an interesting family of domains $\{\mathbb{L}_n: n \geq 2\}$ and studied their geometric aspects. To see this, let $L_n$ be the classical Cartan domain of type $IV$ in $\C^n$ defined as 
\[
L_n=\left\{z \in \mathbb{B}_n : \sqrt{\left(\overset{n}{\underset{j=1}{\sum}}|z_j|^2\right)^2-\left|\overset{n}{\underset{j=1}{\sum}}z_j^2\right|^2}<1-\overset{n}{\underset{j=1}{\sum}}|z_j|^2
 \right\},
\]
where $\mathbb{B}_n=\{(z_1, \dotsc, z_n) \in \C^n: \overset{n}{\underset{j=1}{\sum}}|z_j|^2<1\}$. The map
$\Lambda_n: L_n \to \C^n$ given by $\Lambda_n(z_1, z_2, \dotsc, z_n)=(z_1^2, z_2, \dotsc, z_n)$ is a proper holomorphic mapping with multiplicity of $2$. The domain $\mathbb{L}_n$ is the image of $L_n$ under the map $\Lambda_n$, i.e., $\Lambda_n(L_n)=\mathbb{L}_n$ for $n \geq 2$.  The domains $\mathbb{L}_2$ and $\mathbb{L}_3$ are biholomorphic to $\G_2$ and $\E$, respectively (see Corollary 3.9 in \cite{GhoshI} for explicit formulae of the biholomorphisms). We consider the domain $\mathbb{F}$, introduced in \cite{GhoshI}, and defined as
\[
\F=\left\{(a_{11}, a_{22}, \det(A), a_{12}+a_{21}) \in \C^4: A=\begin{pmatrix}
	a_{11} & a_{12}\\
	a_{21} & a_{22}
\end{pmatrix} \in M_2(\C) \ \text{with} \ \|A\|<1\right\}.
\]
The domain $\F$ is a bounded $(1, 1, 2, 1)$-quasi-balanced domain. Recall that a domain $\Omega\subset\mathbb{C}^n$ is called $(m_1,\cdots,m_n)$-circular (or quasi-circular), where $m_1, \dotsc, m_n$ are relatively prime numbers, if
\[
\begin{pmatrix}\lambda^{m_1}z_1,\cdots,\lambda^{m_n}z_n\end{pmatrix}\in \Omega \qquad \text{for all $\lambda\in\mathbb{T}$ and $z=(z_1,\cdots,z_n)\in \Omega$}.
\]
Also, $\Omega$ is said to be $(m_1,\cdots,m_n)$-\textit{balanced} (or simply \textit{quasi-balanced})If the above relation holds for all $\lambda\in \overline{\mathbb{D}}$. It follows from Corollary 3.9 in \cite{GhoshI} that the domains $\mathbb{L}_4$ and $\F$ are biholomorphic. As mentioned earlier, $\He$ has close connections with $\G_2, \E$ and $\Pe$. Motivated by this, we establish in Section \ref{sec_002} that the domain $\mathbb{F}$ is also closely related to the domains $\mathbb{G}_2$, $\mathbb{E}$ and $\mathbb{P}$. In particular, it is proved that if $(x, a, p, s) \in \F$, then $(s, ax-p) \in \G_2, (x, a, p) \in \E$ and $(a, s, -p) \in \Pe$. Furthermore, it is shown that $\G_2=\{(s, -p): (0, 0, p, s) \in \F\}$ and $\E=\{(x_1, x_2, x_3): (0, x_1, x_2, x_3) \in \F\}$. For $A=(a_{ij})_{i, j}^2 \in M_2(\C)$ with $\|A\|<1$, we have
\[
(a_{11}, a_{22}, \det(A), a_{12}+a_{21}) \in \F \quad \text{and} \quad (a_{21}, a_{11}, a_{22}, \det(A)) \in \He,
\]
where the last statement is proved in Chapter $6$ in \cite{Hexablock}. These similarities between $\F$ and $\He$, together with their respective connections to $\G_2, \E$ and $\Pe$ naturally lead to the following question:

\smallskip 

\noindent \textbf{Q1.} Are the domain $\F$ and $\mathbb H$ biholomorphic ? 

\smallskip

We address this question in Section \ref{sec_003} by studying their respective distinguished boundaries. We mention here that descriptions of the distinguished boundary $\partial_s\F$ of the domain $\F$ and that of the hexablock can be found in \cite{Zwonek, Hexablock}. For our purpose, we obtain additional structural characterizations of $\partial_s \F$ in Theorem \ref{cor_bF}. Capitalizing this, we prove in Theorem \ref{thm_305} that the domains $\F$ and $\He$ are not biholomorphic. The key step in the proof is to show that any biholomorphism between $\F$ and $\He$ extends to a homeomorphism between their closures and, in particular, induces a homeomorphism between their distinguished boundaries, which leads to a contradiction.

\smallskip 

Based on the fact that $\mathbb{L}_2$ and $\mathbb{L}_3$ are biholomorphic to $\G_2$ and $\E$ respectively, Ghosh-Zwonek \cite{GhoshI} raised the following question:

\smallskip 

\noindent \textbf{Q2.} Does the family of domains $\mathbb{L}_n$ arise from a special case of $\mu$-synthesis like the symmetrized polydisc $\mathbb G_n$ and tetrablock $\mathbb E$ ?

\smallskip

In Section~\ref{sec_004}, we address this question for $\mathbb F$, which is biholomorphic to $\mathbb L_4$ and provide an affirmative answer by showing that $\mathbb F$ arises from a suitable $\mu$-synthesis construction. In particular, we show that 
\[
\F=\{(a_{11}, a_{22}, \det(A), a_{12}+a_{21}) : A=(a_{ij}) \in M_2(\C), \mu_E(A)<1\}, 
\]
where $E \subset M_2(\C)$ is the linear subspace given by $E=\left\{\begin{pmatrix}
	z_1 & w \\
	w & z_2
\end{pmatrix} : z_1, z_2, w \in \C \right\}$. Interestingly, it turns out that the choice of $E$ is not unique. As a natural next step, we therefore characterize all such linear subspaces. We thank Professor Włodzimierz Zwonek profusely for drawing our attention to the problems that are investigated in this paper.

\section{The domain $\mathbb{F}$ and its closure}\label{sec_002}

\noindent In this section, we study a domain in $\C^4$ given by
\[
\F=\{\pi_{\F}(B): B \in M_2(\C), \|B\| \leq 1 \},
\] 
where $\pi_{\F}: M_2(\C) \to \C^4, \ \pi_{\F}(B)=(b_{11}, b_{22}, \det(B), b_{12}+b_{21})$ for $B=(b_{ij})_{i, j=1}^2 \in  M_2(\C)$. The authors of \cite{GhoshI} proved that $\F$ is a bounded $(1, 1, 2, 1)$-quasi-balanced-domain. Here, we discuss an alternative characterization of $\F$ and its connection with domains associated with $\mu$-synthesis such as the symmetrized bidisc $\mathbb{G}_2$, the tetrablock $\E$ and the pentablock $\Pe$. To do so, we need the following lemma. 

\begin{lem}\label{lem_201}
	Let $B=\begin{pmatrix}
		x & b_{12}\\
		b_{21} & a
	\end{pmatrix} \in M_2(\C)$ and let $(b_{12}+b_{21}, b_{12}b_{21})=(s, ax-p)$ for $s, p \in \C$. Then 
	\[
	\det(I-B^*B)=1-|a|^2-|x|^2+|p|^2-\frac{|s|^2}{2}-\frac{|s^2-4(ax-p)|}{2}
	\]
and	the following hold.
	\begin{enumerate}
		\item If $|p|<1$, then $\|B\| <1$ if and only if  $\det(I-B^*B)>0$;
		\item If $|p| \leq 1$, then $\|B\| \leq 1$ if and only if  $\displaystyle \det(I-B^*B)\geq 0$.
	\end{enumerate}
\end{lem}

\begin{proof}
	Since $(b_{12}+b_{21}, b_{12}b_{21})=(s, ax-p)$, a routine computation gives that
	\begin{align}\label{eqn_201}
	|s^2-4(ax-p)|=|b_{12}-b_{21}|^2 \quad \text{and} \quad \frac{|s|^2}{2}+\frac{|s^2-4(ax-p)|}{2}=|b_{12}|^2+|b_{21}|^2.
	\end{align}
	Again by a few steps of basic calculations, we have that
	\begin{align*}
		I-B^*B=\begin{pmatrix}
			1-|x_1|^2-|b_{21}|^2 & -\ol{x}b_{12}-a\ol{b}_{21}\\
			-x\ol{b}_{12}-\ol{a}b_{21} & 1-|a|^2-|b_{12}|^2
		\end{pmatrix}
		\end{align*} 
and so, 
\begin{align*}
\det(I-B^*B)
&=(1-|x_1|^2-|b_{21}|^2 )(1-|a|^2-|b_{12}|^2)-|-\ol{x}b_{12}-a\ol{b}_{21}|^2\\
&=1-|a|^2-|x|^2+|p|^2-\frac{|s|^2}{2}-\frac{|s^2-4(ax-p)|}{2},
\end{align*}
where the last equality follows from \eqref{eqn_201}. Again by \eqref{eqn_201}, we have 
\[
\text{tr}(I-B^*B)=2-|a|^2-|x|^2-(|b_{12}|^2+|b_{21}|^2)=2-|a|^2-|x|^2-\frac{|s|^2}{2}-\frac{|s^2-4(ax-p)|}{2}.
\]
Note that $\text{tr}(I-B^*B)=\det(I-B^*B)+(1-|p|^2)$ and so, $\text{tr}(I-B^*B) \geq \det(I-B^*B)$ for $p \in \ol{\mathbb D}$. The desired conclusion now follows from the facts that $\|B\| \leq 1$ if and only if $\text{tr}(I-B^*B), \det(I-B^*B)>0$, and $\|B\| \leq 1$ if and only if $\text{tr}(I-B^*B), \det(I-B^*B)\geq 0$. The proof is now complete.
\end{proof}

Next, we have the following characterization for elements in $\F$.

\begin{prop}\label{prop_302}
	An element $(x, a, p, s) \in \F$ if and only if $(s, ax-p) \in \mathbb{G}_2, p \in \mathbb{D}$ and 
	\[
	1-|a|^2-|x|^2+|p|^2-\frac{|s|^2}{2}-\frac{|s^2-4(ax-p)|}{2}>0.
	\]
\end{prop}

\begin{proof}
Let $(x, a, p, s) \in \F$. Then there exists $B=(b_{ij})_{i, j=1}^2 \in M_2(\C)$ with $\|B\|<1$ such that 
\[
(b_{11}, b_{22}, \det(B), b_{12}+b_{21})=(x, a, p, s).
\]
Therefore, $B=\begin{pmatrix}
	x & b_{12}\\
	b_{21} & a
\end{pmatrix}$ with $(b_{12}+b_{21}, b_{12}b_{21})=(s, ax-p)$. Evidently, $\det(B), b_{12}, b_{21} \in \mathbb D$ since $\|B\|<1$. Thus, $p \in \mathbb{D}$ and $(s, ax-p)=(b_{12}+b_{21}, b_{12}b_{21}) \in \mathbb{G}_2$. By Lemma \ref{lem_201}, we have 
\[
\det(I-B^*B)=1-|a|^2-|x|^2+|p|^2-\frac{|s|^2}{2}-\frac{|s^2-4(ax-p)|}{2}>0.
\]
Conversely, let $(s, ax-p) \in \mathbb{G}_2, p \in \mathbb{D}$ and let
$\displaystyle 1-|a|^2-|x|^2+|p|^2-\frac{|s|^2}{2}-\frac{|s^2-4(ax-p)|}{2}>0$. Since $(s, ax-p) \in \mathbb{G}_2$, one can choose $b_{12}, b_{21} \in \mathbb{D}$ such that $(b_{12}+b_{21}, b_{12}b_{21})=(s, ax-p)$. For $B=\begin{pmatrix}
	x & b_{12}\\
	b_{21} & a
\end{pmatrix}$, it follows from Lemma \ref{lem_201} that $\|B\|<1$ and so, $\pi_{\F}(B)=(x, a, p, s) \in \F$.
\end{proof}

For $(x, a, p, s) \in \F$, we have by Proposition \ref{prop_302} that $(s, ax-p) \in \Gg$.  So, one can choose $(\lm_1, \lm_2) \in \D^2$ such that $(s, ax-p)=(\lm_1+\lm_2, \lm_1\lm_2)$ and so, $(x, a, p, s)=(x, a, p, \lm_1+\lm_2)$. Let $(x_0, a_0, p_0, s_0)=(\lm_1,\lm_2, -p, a+x)$. Then $(s_0, a_0x_0-p_0)=(a+x, ax) \in \Gg$ since $|a|, |x|<1$. Clearly, $|p_0|=|p|<1$. Moreover, it is not difficult to see that
\[
1-|a_0|^2-|x_0|^2+|p_0|^2-\frac{|s_0|^2}{2}-\frac{|s_0^2-4(a_0x_0-p_0)|}{2}=		1-|a|^2-|x|^2+|p|^2-\frac{|s|^2}{2}-\frac{|s^2-4(ax-p)|}{2},
\]
which, by Proposition \ref{prop_302}, is strictly positive. Again by Proposition \ref{prop_302}, $(\lm_1,\lm_2, -p, a+x) \in \F$ if $(a, x, p, \lm_1+\lm_2) \in \F$. The converse follows by similar arguments. So, we have the following result.

\begin{cor}
	For $(x, a, p, s) \in \C^4$, let $(s, ax-p)=(\lm_1+\lm_2, \lm_1\lm_2) \in \Gg$. Then $(x, a, p, s) \in \F$ if and only if $(\lm_1,\lm_2, -p, a+x) \in \F$.
\end{cor}

There is an analogue of Proposition \ref{prop_302} for the set $\ol{\F}$ as the next result shows. 
\begin{prop}\label{prop_closedF}
	Let $(x, a, p, s) \in \C^4$. Then the following are equivalent:
	\begin{enumerate}
		\item $(x, a, p, s) \in \ol{\F}$;
		\item there exists $A=(a_{ij})_{i, j=1}^2 \in M_2(\C)$ with $\|A\| \leq 1$ such that $\pi_{\F}(A)=(x, a, p, s)$;
		\item $(s, ax-p) \in \Gamma, p \in \ol{\mathbb{D}}$ and 
		$\displaystyle 
		1-|a|^2-|x|^2+|p|^2-\frac{|s|^2}{2}-\frac{|s^2-4(ax-p)|}{2} \geq 0.
		$
		\end{enumerate} 
\end{prop}

\begin{proof}  
	It is straightforward to verify that the proof of $(2) \iff (3)$ proceeds along the same lines as that of Proposition \ref{prop_302}. So, we prove $(1) \implies (2)$ and $(2) \implies (1)$.
	
	\medskip 
	
	\noindent $(1) \implies (2)$. Pick a sequence $\alpha_n \in \F$ such that $\alpha_n \to (x, a, p, s)$. Then one can choose $A_n \in M_2(\C)$ with $\|A_n\|<1$ for every $n \in \N$ such that $\pi_{\F}(A_n)=\alpha_n$. Passing to a convergent subsequence, if necessary, we have that $\pi_\F(A)=\lim_{n \to \infty}\pi_\F(A_n)=\lim_{n \to \infty} \alpha_n=(x, a, p, s)$.
	
	\medskip 
	
\noindent 	$(2) \implies (1)$. Let $\alpha=\pi_\F(A)$ for some $A=(a_{ij})_{i, j=1}^2 \in M_2(\C)$ with $\|A\| \leq 1$. For $r \in (0, 1)$, we have $\|rA\|<1$ and so, $\pi_\F(rA)=(ra_{11}, ra_{22}, r^2\det(A), ra_{12}+ra_{21}) \in \F$. It is easy to see that $\lim_{r\to 1}\pi_\F(rA) = \pi_{\F}(A)=\alpha$ and thus, $\alpha \in \ol{\F}$. The proof is now complete.
\end{proof}

The next lemma follows from Proposition \ref{prop_closedF}, and the fact that $\F$ is $(1, 1, 2, 1)$-quasi-balanced.

\begin{lem}\label{lem_Fr}
	Let $(x_1, x_2, x_3, x_4) \in \overline{\mathbb{F}}$ and let $r \in (0,1)$. Then $(rx_1, rx_2, r^2x_3, rx_4) \in \mathbb{F}$. 
\end{lem}

The domains $\Gg$, $\E$ and $\Pe$ exhibit a particularly close relationship with the hexablock $\He$, as explored in Chapter 10 of \cite{Hexablock}. In a similar spirit, we establish analogous connection of $\F$ with the domains $\G_2$, $\E$ and $\Pe$. To do so, we recall a few useful results from the literature in this direction. 

\subsection*{The symmetrized bidisc} The symmetrized bidisc is a bounded domain in $\C^2$ which was introduced by Agler and Young \cite{AglerYoung} in 1999. It is defined as 
\[
\Gg=\{(z_1+z_2, z_1z_2) \in \C^2 : z_1, z_2 \in \D \},
\]
and its closure is given by $\Gamma=\{(z_1+z_2, z_1z_2) \in \C^2 : z_1, z_2 \in \overline{\D} \}$. We mention a few important results in this direction from the literature \cite{AglerYoung, Agler2003, AglerIII, Agler2004,  Tirtha_Pal, Costara2004, Costara2005}.

\begin{thm}[\cite{AglerIII}, Theorem 2.1]\label{thmG_2}
	For  $(s, p) \in \C^2$, the following are equivalent:
	\begin{enumerate}
		\item $(s, p) \in \Gg$;
		\item $|s-\overline{s}p|<1-|p|^2$;
		\item $2|s-\overline{s}p|+|s^2-4p|<4-|s|^2$;
		\item there exists $A \in M_2(\C)$ with $\|A\|<1$ such that $s=\text{tr}(A)$ and $p=\det(A)$.
	\end{enumerate}
\end{thm}	

For a bounded domain $\Omega$ in $\C^n$, let $A(\Omega)$ be the algebra of continuous functions on $\overline{\Omega}$ which are holomorphic in $\Omega$. The \textit{distinguished boundary} $b\Omega$ (or the Shilov boundary $\pr_s \Omega$) of $\Omega$ is the smallest closed subset of $\overline{\Omega}$ such that every function in $A(\Omega)$ attains its maximum modulus on $b\Omega$. Agler and Young \cite{AglerIII} proved that $b\Gamma=\{(z_1+z_2, z_1z_2) : z_1, z_2 \in \T\}$. The next result gives criteria for elements in $b\Gamma$. 

\begin{thm}[\cite{Tirtha_Pal}, Theorem 1.3]\label{thm_bGam}
	Let $(s, p) \in \C^2$. The following are equivalent:
	\begin{enumerate}
		\item $(s, p) \in b\Gamma$;
		\item $|p|=1, s=\overline{s}p$ and $|s| \leq 2$;
		\item $(s, p) \in \Gamma$ and $|p|=1$;
		\item $|p|=1$ and there exists $\beta \in \T$ such that $s=\beta+\overline{\beta}p$.
	\end{enumerate}
\end{thm}

\subsection*{The tetrablock} The symmetrized bidisc and its connection with $\mu$-synthesis has led to various other domains corresponding to different cases of $\mu$-synthesis. In this direction, Abouhajar et al. \cite{Abouhajar} introduced a domain in 2007 which is defined as
\[
\E=\{(x_1, x_2, x_3) \in \C^3 \ : \ 1-x_1z_1-x_2z_2+x_3z_1z_2 \ne 0 \ \text{whenever} \ |z_1| \leq 1, |z_2| \leq 1 \}.
\]
The set $\E$ is a bounded domain in $\C^3$ which is referred to as the tetrablock and its closure is denoted by $\Ebar$. The domain tetrablock has turned out to be a domain of independent interest in several complex variables \cite{Abouhajar, Alsalhi, EdigarianI, Kosi, zwonek1, NT, Young, Zwonek} and operator theory \cite{Ball_Sau, Tirtha, Bisai-Pal1, Bisai-Pal2, Pal_E1, Pal_E2, Pal_E3}. We recall here a few important facts about the tetrablock useful for our purpose.

\begin{thm}[\cite{Abouhajar}, Theorem 2.2]\label{tetrablock}
	For $(x_1, x_2, x_3) \in \C^3$, the following are equivalent:
	\begin{enumerate}
		\item $(x_1, x_2, x_3) \in \E$ ;
		\item $|x_1|^2+|x_2-\overline{x}_1x_3|+|x_1x_2-x_3|<1$;
		\item there is a $2 \times 2$ matrix $A=(a_{ij})$ such that $\|A\|<1$ and $x=(a_{11}, a_{22}, \det(A))$;
		\item $1-|x_1|^2-|x_2|^2+|x_3|^2>2|x_1x_2-x_3|$ and if $x_1x_2=x_3$ then, in addition, $|x_1|+|x_2|<2$.
	\end{enumerate}	
\end{thm}

We recall from \cite{Abouhajar} the description of the distinguished boundary of the tetrablock.

\begin{thm}[\cite{Abouhajar}, Theorem 7.1]\label{thm6.1}
	For $x=(x_1, x_2, x_3) \in \C^3$, the following are equivalent:
	\begin{enumerate}
		\item $x_1=\overline{x}_2x_3, |x_3|=1$ and $|x_2|\leq 1$;
		\item there is a $2 \times 2$ unitary matrix $U=(u_{ij})$ such that $x=(u_{11}, u_{22}, \det(U))$; 
		\item $x \in b\E$;
		\item $x \in \overline{\E}$ and $|x_3|=1$. 
	\end{enumerate}	
\end{thm}

\subsection*{The pentablock} The pentablock is a domain introduced in 2015 by Agler et al. \cite{AglerIV}. Like the symmetrized bidisc and tetrablock, it naturally appeared in the study of the $\mu$-synthesis problem. The pentablock $\Pe$ is defined as 
\[
\Pe=\{(a_{21}, \text{tr}(A), \det(A)) : A=(a_{ij})_{i, j=1}^2 \in M_2(\C), \|A\|<1 \},
\] 
which is a bounded domain in $\C^3$. Its closure is denoted by $\Pbar$. The pentablock has many interesting geometric properties and has therefore attracted a lot of attention recently. An interested reader is referred to the works \cite{AglerIV, Alsheri, Jindal, JindalII, KosinskiII, PalN, Guicong, GuicongII}. We recall from literature various characterizations and properties of the pentablock.

\begin{thm}[\cite{AglerIV}, Theorem 5.2]\label{pentablock}
	Let $a \in \C$ and let $(s, p)=(\lm_1+\lm_2, \lm_1\lm_2) \in \Gg$. Then the following are equivalent:	
	\begin{enumerate}
		\item $(a, s, p) \in \mathbb{P}$;
		\item $
		\sup_{z \in \D} \bigg|\displaystyle\frac{a(1-|z|^2)}{1-sz+pz^2}\bigg|<1$;
		\item $
		|a|< \frac{1}{2}|1-\overline{\lambda}_2\lambda_1|+\frac{1}{2}\sqrt{(1-|\lambda_1|^2)(1-|\lambda_2|^2)}$.
	\end{enumerate}
\end{thm}

The following result gives a description of $b\Pe$, the distinguished boundary of the pentablock.

\begin{thm}[\cite{AglerIV}, Theorem  8.4]\label{thmbP}
	$(a, s, p) \in b\Pe$ if and only if $(s, p) \in b\Gamma$ and $|a|^2+|s|^2\slash 4=1$.
\end{thm} 	

We  now present the following result showing the interplay among the domains $\F, \G_2, \E$ and $\Pe$.

\begin{thm}\label{thm_305}
	Let $(x, a, p, s) \in \F$. Then $(s, ax-p) \in \G_2, (x, a, p) \in \E$ and $(a, s, -p) \in \Pe$.
\end{thm}

\begin{proof}
	Since $(x, a, p, s) \in \F$, there exists $A=(a_{ij})_{i, j=1}^2 \in M_2(\C)$ with $\|A\| < 1$ such that $\pi_{\F}(A)=(a_{11}, a_{22}, \det(A), a_{12}+a_{21})=(x, a, p, s)$. Define 
	\[
	U=\begin{pmatrix}
		0 & 1 \\
		1 & 0
	\end{pmatrix} \quad \text{and so,} \quad AU=\begin{pmatrix}
	a_{12} & a_{11} \\
	a_{22} & a_{21}
	\end{pmatrix}.
	\]
	Since $U$ is a unitary, we have that $\|AU\|=\|A\|<1$. By the definition of  $\Pe$ that $(a_{22}, a_{12}+a_{21}, -\det(A))=(a, s, -p) \in \Pe$. The fact that $(x, a, p) \in \E$ follows from the definition of $\F$ and Part $(3)$ of Theorem \ref{tetrablock}. Moreover, we have by Proposition \ref{prop_302} that $(s, ax-p) \in \G_2$, which completes the proof.  
\end{proof}



\begin{prop}\label{prop_306}
	A point $(x_1, x_2, x_3, 0)$ belongs to $\F$ if and only if $(0, x_1, x_2, x_3) \in \He$, which in turn is equivalent to $(x_1, x_2, x_3) \in \E$.
\end{prop} 

\begin{proof}
	It follows from Theorem 6.1 in \cite{Hexablock} that $(0, x_1, x_2, x_3) \in \He$ if and only if $(x_1, x_2, x_3) \in \E$. Moreover, it follows from Theorem \ref{thm_305} that $(x_1, x_2, x_3) \in \E$ if $(x_1, x_2, x_3, 0) \in \F$. Let $(x_1, x_2, x_3) \in \E$ and let $(x, a, p, s)=(x_1, x_2, x_3, 0)$. We have by Proposition \ref{prop_302} that 
	$(x_1, x_2, x_3, 0) \in \F$ if and only if
	\begin{align}\label{eqn_302}
		(0, x_1x_2-x_3) \in \G_2, \quad  |x_3|<1 \quad \text{and} \quad 1-|x_2|^2-|x_1|^2+|x_3|^2-2|x_1x_2-x_3|^2>0.
	\end{align}
	By Theorem \ref{thmG_2}, $(0, x_1x_2-x_3) \in \G_2$ if and only if $|x_1x_2-x_3|<1$, which follows directly from Part (2) of Theorem \ref{tetrablock}. Since $(x_1, x_2, x_3) \in \E$, we have $|x_3|<1$. Finally, the last inequality in \eqref{eqn_302} follows from Part (4) of Theorem \ref{tetrablock}. The proof is now complete.
\end{proof}

We have an analogous characterization for points in the symmetrized bidisc. 

\begin{prop}\label{prop_307}
	A point $(0, 0, p, s) \in \F$ if and only if $(s, -p) \in \Gg$.
\end{prop}

\begin{proof}
	We have by Theorem \ref{thm_305} that $(0, s, -p) \in \Pe$ and so, by Theorem \ref{pentablock}, $(s, -p) \in \Gg$. Conversely, assume that $(s, -p) \in \Gg$. One can choose $(\lm_1, \lm_2) \in \D^2$ such that $(s, -p)=(\lm_1+\lm_2, \lm_1\lm_2)$. A simple calculation shows that $s^2+4p=(\lm_1-\lm_2)^2$ and so,
	\[
	1+|p|^2-\frac{|s|^2}{2}-\frac{|s^2+4p|}{2}=1+|\lm_1\lm_2|^2-\frac{|\lm_1+\lm_2|^2}{2}-\frac{|\lm_1-\lm_2|^2}{2}=\sqrt{(1-|\lm_1|^2)(1-|\lm_2|^2)}>0.
	\]
	By Proposition \ref{prop_302}, $(0, 0, p, s) \in \F$ and the proof is now complete.	
\end{proof}

It follows from Theorem \ref{thm_305} that $(a, s, -p) \in \Pe$ if $(0, a, p, s) \in \F$. Motivated by Propositions \ref{prop_306} and \ref{prop_307}, one may ask if converse to this statement holds, i.e., if $(a, s, -p) \in \Pe$, then $(0, a, s, -p) \in \F$. The following example shows the general statement does not hold. 

\begin{eg}\label{eg_309}
	Let $r \in (0, 1)$ and let $(a, s, p)=\left(1-r^2 \slash 2, 0, r^2\right)$. Let $(\lm_1, \lm_2)=(r, -r)$. Clearly, $(s, -p)=(\lm_1+ \lm_2, \lm_1\lm_2)$ and thus, $(s, -p) \in \Gg$. By Part (3) of Theorem \ref{pentablock}, $(a, s, -p) \in \Pe$ since 
	\[
	\frac{1}{2}|1-\overline{\lambda}_2\lambda_1|+\frac{1}{2}\sqrt{(1-|\lambda_1|^2)(1-|\lambda_2|^2)}=1>|a|=1-\frac{r^2}{2}.
	\]
	A simple calculation gives that
	\begin{align*}
		1-|a|^2+|p|^2-\frac{|s|^2}{2}-\frac{|s^2-4(ax-p)|}{2}
		&=1-\left(1-\frac{r^2}{2}\right)^2+r^4-2r^2=(1-r^2)^2-\left(1-\frac{r^2}{2}\right)^2<0
	\end{align*}
and by Proposition \ref{prop_302}, $(0, a, p, s) \notin \F$. 	\qed 
\end{eg}

The following result is an immediate corollary to Theorem \ref{thm_305}, Propositions \ref{prop_306} and \ref{prop_307}.
\begin{cor}\label{cor_309}
	Let $(x, a, p, s) \in \C^4$. Then the following statements hold.
	\begin{enumerate}
		\item 	If $(x, a, p, s) \in \ol{\F}$, then $(s, ax-p) \in \Gamma, (s, -p) \in \Gamma, (x, a, p) \in \Ebar$ and $(a, s, -p) \in \Pbar$. \vspace{0.1cm}
		\item $(x, a, p) \in \ol{\E}$ if and only if $(x, a, p, 0) \in \ol{\F}$. \vspace{0.1cm}
		\item $(s, p) \in \Gamma$ if and only if $(0, 0, -p, s) \in \ol{\F}$. 
	\end{enumerate}
\end{cor}

We now present an example showing that the converse to Part $(1)$ of Corollary \ref{cor_309} need not hold, i.e., $(x, a, p, s)$ may fail to belong to $\ol{\F}$ even if $(s, ax-p) \in \Gamma, (x, a, p) \in \Ebar$ and $(a, s, -p) \in \Pbar$.

\begin{eg}\label{eg_310}
	For $r \in (0, 1)$, define $(x_r, a_r, p_r, s_r)=\left(0, 1-r^2 \slash 2, 0, r^2\right)$. We have by Example \ref{eg_309} that $(a_r, s_r, -p_r) \in \Pe$ and $(s_r, a_rx_r-p_r)=(s_r, -p_r) \in \Gg$. Also, by Part $(2)$ of Theorem \ref{tetrablock}, $(x_r, a_r, p_r) \in \E$. Let $(x, a, p, s)=(0, 1\slash 2, 0, 1)$. Clearly, $(x, a, p, s)$ is a limit point of $(x_r, a_r, p_r, s_r)$ as $r \to 1$. Hence, $(s, ax-p) \in \Gamma, (x, a, p) \in \Ebar$ and $(a, s, -p) \in \Pbar$.
	A simple calculation gives that
	\begin{align*}
		1-|a|^2+|p|^2-\frac{|s|^2}{2}-\frac{|s^2-4(ax-p)|}{2}
		&=-\frac{1}{4}<0
	\end{align*}
	and so, by Proposition \ref{prop_closedF}, $(x, a, p, s) \notin \ol{\F}$.	\qed
\end{eg}

\subsection*{The normed Hexablock $\QN$ and the domain $\F$.} As mentioned in Section \ref{sec_01}, the domains $\Gg$, $\E$ and $\Pe$ arise while studying certain instances of $\mu$-synthesis problem. As discussed earlier, these domains can be written as 
\begin{align*}
	\Gg&=\left\{(a_{11}+a_{22}, \det(A)): A=(a_{ij}) \in M_2(\C), \|A\|<1 \right\},\\ 
	\E&=\left\{(a_{11}, a_{22}, \det(A)) : A=(a_{ij}) \in M_2(\C), \|A\|<1 \right\},\\
	\Pe&=\left\{(a_{21}, a_{11}+a_{22}, \det(A)) : A=(a_{ij}) \in M_2(\C), \|A\|<1\right\}.
\end{align*}
Motivated by the above descriptions of $\Gg, \E$ and $\Pe$, the authors of \cite{Hexablock} studied the set given by
\[
\QN=\left\{(a_{21}, a_{11}, a_{22}, \det(A)) : A=(a_{ij}) \in M_2(\C), \|A\|<1 \right\},
\]
and refer to $\QN$ as the \textit{normed hexablock}. Furthermore, $\QN$ is not a domain in $\C^4$, and the interior of the polynomial convex hull of $\ol{\He}_N$ equals the hexablock $\He$. So, the hexablock and the normed hexablock are closely related. Also, it was proved in \cite{Hexablock} that $\QN$ is strictly contained in $\He$. The next result establishes the connection with the normed hexablock $\QN$ and the domain $\F$.

\begin{prop}\label{prop_310}
	For $(a, x_1, x_2, x_3) \in \C^4$, the following holds. 
	\begin{enumerate}
		\item If $(0, x_1, x_2, x_3) \in \QN$, then $(x_1, x_2, x_3, 0) \in \F$.
		\item If $a \ne 0$, then $(a, x_1, x_2, x_3) \in \QN$ if and only if $\displaystyle \left(x_1, x_2, x_3, a+\frac{x_1x_2-x_3}{a}\right) \in \F$
	\end{enumerate} 
\end{prop}

\begin{proof}
	For $a=0$, let us define $A_0=\begin{pmatrix}
		x_1 & 0\\
		0 & x_2\\
	\end{pmatrix}$. We have by Theorem 5.6 in \cite{Hexablock} that
	\begin{align*}
	(0, x_1, x_2, x_3) \in \QN 	
	 \implies (x_1, x_2, x_3) \in \E \ \text{and} \ x_1x_2=x_3 
	 \implies \|A_0\|=\max\{|x_1|, |x_2|\} <1  
	\end{align*} 
and so, $\pi_{\F}(A_0)=(x_1, x_2, x_3, 0) \in \F$. Let $a\ne 0$ and let $w \in \C$ with $w^2=x_1x_2-x_3$. Define 
\[
B=\begin{pmatrix}
	b_{11} & b_{12}\\
	b_{21} & b_{22}\\
\end{pmatrix}=\begin{pmatrix}
	x_1 & w^2\slash a\\
	a & x_2\\
\end{pmatrix}.
\]
 If $(a, x_1, x_2, x_3) \in \QN$, then we have by Lemma 5.4 and Theorem 5.6 in \cite{Hexablock} that $(x_1, x_2, x_3) \in \E$ and $\|B\|<1$. Thus, $\pi_\F(B)=\displaystyle \left(x_1, x_2, x_3, a+\frac{x_1x_2-x_3}{a}\right) \in \F$. Conversely, assume that 
\[
a \ne 0 \quad \text{and} \quad (x, a_0, p, s)=\left(x_1, x_2, x_3, a+\frac{x_1x_2-x_3}{a}\right) \in \F.
\]
By definition, there exists $A=(a_{ij}) \in M_2(\C)$ with $\|A\|<1$ such that $\pi_\F(A)=(x, a_0, p, s)=(a_{11}, a_{22}, \det(A), a_{21}+a_{12})$. In particular, we have
\[
A=\begin{pmatrix}
	x & a_{12}\\
	 a_{21} & a_0
\end{pmatrix} \quad \text{with} \quad (a_{12}+a_{21}, a_{12}a_{21})=(s, a_0x-p)=\left(a+\frac{w^2}{a}, w^2\right)=(b_{12}+b_{21}, b_{12}b_{21}).
\]
It follows from Lemma \ref{lem_201} that 
\begin{align*}
\det(I-A^*A)=\det(I-B^*B)= 1-|a_0|^2-|x|^2+|p|^2-\frac{|s|^2}{2}-\frac{|s^2-4(a_0x-p)|}{2}.
\end{align*}
Since $\|A\|<1$, we have by Lemma \ref{lem_201} that $\det(I-A^*A)=\det(I-B^*B)>0$ and so, $\|B\|<1$. Hence, $(b_{21}, b_{11}, b_{22}, \det(B))=(a, x_1, x_2, x_3) \in \QN$. The proof is now complete.
\end{proof}

Note that the converse in Part (1) of the above theorem need not hold in general. For example, $(x_1, x_2, x_3, a)=(0, 0, r^2, 0) \in \F$ for all $r \in (0, 1)$. In fact, we have
\[
(0, 0, r^2, 0)=\pi_\F(A), \quad \text{where}  \quad A=\begin{pmatrix}
	0 & r \\
	-r & 0
\end{pmatrix} \quad \text{which satisfies $\|A\|<1$}.
\]
However, $(a, x_1, x_2, x_3)=(0, 0, 0, r^2) \notin \QN$ since $x_1x_2 \ne x_3$ (see Theorem 5.6 in \cite{Hexablock}). In fact, Proposition \ref{prop_306} characterizes the points of the form $(x_1, x_2, x_3, 0) \in \C^4$ belong to the domain $\F$.

\medskip 

We now present an alternative characterization of the normed hexablock $\QN$ via the domain $\F$.

\begin{cor} The interior of $\QN$ and the normed hexablock $\QN$ can be written as
	\begin{align*}	
		\text{int}(\QN)&=\left\{(a, x_1, x_2, x_3) \in \C^4: a \ne 0, \displaystyle \left(x_1, x_2, x_3, a+\frac{x_1x_2-x_3}{a}\right) \in \F \right\} \text{and} \\ 
		\QN&=\text{int}(\QN) \cup \{(0, x_1, x_2, x_1x_2) \in \C^4 : (0, x_1, x_2, x_1x_2) \in \F\}.
	\end{align*} 
\end{cor}

\begin{proof}
	It follows from Proposition 5.10 in \cite{Hexablock} that $\text{int}(\QN)=\{(a, x_1, x_2, x_3) \in \QN : a \ne 0\}$ and so, the first equality for $\text{int}(\QN)$ in the statement of the theorem follows from Proposition \ref{prop_310}. The rest of the conclusion follows from Proposition \ref{prop_306}, and the fact that $(0, x_1, x_2, x_3) \in \QN$ if and only if $x_3=x_1x_2$ and $(x_1, x_2, x_3) \in \E$ (see Theorem 5.6 in \cite{Hexablock}). The proof is complete.
\end{proof}

	\section{The distinguished (Shilov) boundary of $\mathbb{F}$}\label{sec_003}

\noindent For $z \in \mathbb{C}^n$, the Euclidean norm of $z$ is denoted by $\|z\|$ and $z \bullet z=z_1^2+\dotsc+z_n^2$. We denote the unit ball in $\mathbb{C}^n$ by $\mathbb{B}_n=\{z \in \mathbb{C}^n: \|z\|<1\}$. For $n \geq 1$, the Lie Ball $L_n$ is defined as
\[
L_n=\left\{z \in \mathbb{B}_n : 2\|z\|^2-|z\bullet z|^2<1\right\},
\]
which is a convex and balanced domain. The Shilov boundary of $L_n$ is given by
\[
\partial_sL_n=\left\{e^{i\theta}x=(e^{i\theta}x_1, \dotsc, e^{i\theta}x_n): \theta \in \mathbb{R}, \ x \in \mathbb{R}^n,\ x_1^2+\dotsc+ x_n^2=1\right\}.
\]	
Consider the proper holomorphic map defined as
\[
\Lambda_n: L_n \to \mathbb{C}^n, \quad \Lambda_n(z_1, z_2, \dotsc, z_n)=(z_1^2, z_2, \dotsc, z_n).
\]	
The domain $\mathbb{L}_n$ is defined as $L_n$ under the map $\Lambda_n$, i.e., $\mathbb{L}_n=\Lambda_n(L_n)$. In particular, $\mathbb{L}_4=\Lambda_4(L_4)$. Moreover, the domain $\mathbb{L}_4$ is a $(2, 1, 1, 1)$-quasi-balanced domain. It follows from Proposition 4.1 in \cite{GhoshI} that the Shilov boundary of $\mathbb{L}_4$ is given by
\begin{align}\label{eqn_bL4}
	\partial_s \mathbb{L}_4=\Lambda_4(\partial_s L_4)=\left\{ (e^{2i\theta}x_1^2, e^{i\theta}x_2, e^{i\theta}x_3, e^{i\theta}x_4) : \theta, x_1, x_2, x_3, x_4 \in \mathbb{R}, \ x_1^2+ x_2^2+x_3^2+x_4^2=1   \right\}.
\end{align}
Again, we have by Corollary 3.9 in \cite{GhoshI} that the following map 
\begin{align}\label{eqn_L4_F}
	f: \mathbb{L}_4 \to \mathbb{F}, \quad f(w_1, w_2, w_3, w_4)=(w_3+iw_4, -w_3+iw_4, -w_2^2-w_3^2-w_4^2-w_1, 2w_2)
\end{align}
defines a biholomorphism. In this section, we explicitly describe the Shilov boundary of the domain $\F$ using the biholomorphism in \eqref{eqn_L4_F}, and the fact that both $\mathbb{L}_4$ and $\F$ are quasi-balanced domains. Furthermore, we show that the domains $\F$ and $\He$ are not biholomorphic to each other for which it suffices to show that their Shilov boundaries are not homeomorphic. We begin with the following result due to Kosi\'{n}ski \cite{Kosi}, which plays a crucial role in this section. 

\begin{lem}[\cite{Kosi}, Lemma 6]\label{extension}
	Let $D, G$ be bounded domains in $\mathbb{C}^n$. Suppose that $G$ is $( m_1, \cdots , m_n)$-circular and that it contains the origin. Furthermore, we assume that the Bergman kernel function $K_{D}(z,\bar{\xi})~(z,\xi\in D)$ associated with $D$ satisfies the following property: for any open, relatively compact subset $E$ of $D$, there is an open set $U=U(E)$ that contains $\overline{D}$ such that $K_D(z,\bar{\xi})$ extends to be holomorphic on $U$ as a function of $z$ for each $\xi \in E$. Then, any proper holomorphic mapping $f: D \to G$ extends holomorphically to a neighborhood of $\overline {D}$.
\end{lem}

We also recall from \cite{Kosi, Nik}  that the \textit{Minkowski functional} for a given $(m_1, \dotsc, m_n)$-quasi-balanced domain $D$ in $\mathbb{C}^n$is defined as
\[
M_D(x):=\inf \left\{t>0: (t^{-{m_1}}x_1, \dotsc, t^{-{m_n}}x_n) \in D \right\}, \quad x=(x_1, \dotsc, x_n) \in \mathbb{C}^n.
\]
For $r>0$, we denote by $D_r=\{x \in \mathbb{C}^n: M_D(x)<r\}$. The Minkowski functional has several interesting properties. In particular, we have
\begin{align}
	\label{Mink_I} M_D(\alpha^{m_1} x_1^m, \dotsc, \alpha^{m_n}x_n)&=|\alpha|M_D(x), \quad x=(x_1, \dotsc, x_n) \in \mathbb{C}^n, \alpha \in \mathbb{C} \\
	\label{Mink_II} D&=\{x \in \mathbb{C}^n: M_D(x)<1\}.
\end{align} 
It follows from  Remark 7 in \cite{Kosi} that the Minkowski functional for a bounded quasi-balanced domain $D$ is continuous if and only if $D$ is relatively compact in $D_{1\slash r}$ for every $r \in (0,1)$.

\medskip

 Since both $\mathbb{L}_4$ and $\mathbb{F}$ are both quasi-balanced bounded domains in $\mathbb{C}^4$ containing the origin, it follows from Lemma \ref{extension}, and Remark 7 in \cite{Kosi} that $f$ extends to a biholomorphism from $\overline{\mathbb{L}}_4$ onto $\overline{\mathbb{F}}$. For $r \in (0,1)$, we define
\[
\mathbb{F}_{1\slash r}=\left\{x=(x_1, x_2, x_3, x_4) \in \mathbb{C}^4 : M_{\mathbb{F}}(x_1, x_2, x_3, x_4)<1\slash r \right\}.
\]
It follows from \eqref{Mink_I} and \eqref{Mink_II} that $\mathbb{F}_{1\slash r}=\left\{(x_1, x_2, x_3, x_4) \in \mathbb{C}^4 : (rx_1, rx_2, r^2x_3, rx_4) \in \mathbb{F} \right\}$ and so, $\mathbb{F}_{1\slash r}$ is a domain in $\mathbb{C}^4$. We have by Lemma \ref{lem_Fr} that $\mathbb{F}$ is relatively compact in $\mathbb{F}_{1\slash r}$ for every $r \in (0,1)$. An application of Part $(a)$ of Corollary 4 and Remark 7 in \cite{Kosi} yields that the map $f: \mathbb{L}_4 \to \mathbb{F}$ (as in \eqref{eqn_L4_F}) maps $\partial_s \mathbb{L}_4$ onto $\partial_s\mathbb{F}$. Thus, we arrive at the following result.

\begin{thm}\label{thm_bF1}
	The Shilov boundary of $\F$ is given by
	$
	\pr_s\F=f(\pr_s \LL_4),
	$
	where $f$ is the map as in \eqref{eqn_L4_F}. 
\end{thm}

We provide below various other characterizations of the Shilov boundary of $\F$.

\begin{thm}\label{cor_bF}
Let $(x, a, p, s) \in \C^4$. Then the following are equivalent: 
\begin{enumerate}
	\item $(x, a, p, s) \in \pr_s \F$;
	\item there exist real numbers $\theta, x_1, x_2, x_3, x_4$ with $x_1^2+x_2^2+x_3^2+x_4^2=1$ such that 
	\[
	(x, a, p, s)=\left(e^{i\theta}(x_3+ix_4), \ -e^{i\theta}(x_3-ix_4),\ -e^{2i\theta}, \ 2e^{i\theta}x_2 \right);
	\]
	\item there exist $(z, w) \in \ol{\mathbb{B}}_2$ and $\eta \in \mathbb{T}$ such that
	$
	(x, a, p, s)=\left(\ol{z}, -\eta z, -\eta, w+\ol{w}\eta \right);
	$
	\item $(x, a, p) \in b\E$ and there exists $w \in \C$ such that $(x, w) \in \ol{\mathbb{B}}_2$ and $s=w-\overline{w}p$.
\end{enumerate}	
\end{thm}

\begin{proof}
	We have by \eqref{eqn_bL4} that 
	\[
	\partial_s \mathbb{L}_4=\left\{ (e^{2i\theta}x_1^2, e^{i\theta}x_2, e^{i\theta}x_3, e^{i\theta}x_4) : \theta \in \mathbb{R},  (x_1, x_2, x_3, x_4) \in \mathbb{R}^4, \ x_1^2+ x_2^2+x_3^2+x_4^2=1   \right\}.
	\]
	For $(x_1, x_2, x_3, x_4)\in \R^4$ with $x_1^2+ x_2^2+x_3^2+x_4^2=1$,  it follows from \eqref{eqn_L4_F} that 
	\[
	f(e^{2i\theta}x_1^2, e^{i\theta}x_2, e^{i\theta}x_3, e^{i\theta}x_4)=\left(e^{i\theta}(x_3+ix_4), \ -e^{i\theta}(x_3-ix_4),\ -e^{2i\theta}, \ 2e^{i\theta}x_2 \right).
	\]
By Theorem \ref{thm_bF1}, $(1)$ and $(2)$ are equivalent conditions. We prove $(2) \implies (3) \implies (4) \implies (2)$.

\medskip 

\noindent $(2) \implies (3)$. Suppose there exist real numbers $\theta, x_1, x_2, x_3, x_4$ with $x_1^2+x_2^2+x_3^2+x_4^2=1$ such that 
\[
(x, a, p, s)=\left(e^{i\theta}(x_3+ix_4), \ -e^{i\theta}(x_3-ix_4),\ -e^{2i\theta}, \ 2e^{i\theta}x_2 \right).
\]
Define $(z, w, \eta)=\left(e^{-i\theta}(x_3-ix_4), \ e^{i\theta}x_2, \ e^{2i\theta}\right)$. It is not difficult to see that $|z|^2+|w|^2=1-x_1^2 \leq 1$ and $(x, a, p, s)=\left(\ol{z}, -\eta z, -\eta, w+\ol{w}\eta \right)$.

\medskip 

\noindent $(3) \implies (4)$. Suppose there exist $(z, w) \in \ol{\mathbb{B}}_2$ and $\eta \in \mathbb{T}$ such that $(x, a, p, s)=\left(\ol{z}, -\eta z, -\eta, w+\ol{w}\eta \right)$. It is easy to see that $x=\overline{a}p, |p|=1$ and $|a| \leq 1$ and so, by Theorem \ref{thm6.1}, $(x, a, p) \in b\E$. Clearly, $(x, w)=(\overline{z}, w) \in \ol{\mathbb{B}}_2$ and $s=w+\overline{w}\eta=w-\overline{w}p$. 

\medskip

\noindent $(4) \implies (2)$. Assume that $(x, a, p) \in b\E$ and there exists $w \in \C$ such that $(x, w) \in \ol{\mathbb{B}}_2$ and $s=w-\overline{w}p$. By Theorem \ref{thm6.1}, there exist $z \in \ol{\D}$ and $\eta \in \mathbb{T}$ such that $(x, a, p)=\left(\ol{z}, -\eta z, -\eta\right)$. Then $s=w+\ol{w}\eta$. One can choose $\theta, x_3, x_4 \in \mathbb{R}$ such that $\eta=e^{2i\theta}$ and $e^{i\theta}z=x_3-ix_4$. Now, we put $x_2=\text{Re}(e^{-i\theta}w)$. Evidently, $x_2^2+x_3^2+x_4^2=\left(\text{Re}(e^{-i\theta}w)\right)^2+|z|^2 \leq |w|^2+|z|^2 \leq 1$.
Choose $x_1=(1-x_2^2-x_3^2-x_4^2)^{1\slash 2}$. A routine computation gives that 
\[
(x, a, p, s)=\left(\ol{z}, -\eta z, -\eta, w+\ol{w}\eta \right)=\left(e^{i\theta}(x_3+ix_4), \ -e^{i\theta}(x_3-ix_4),\ -e^{2i\theta}, \ 2e^{i\theta}x_2 \right).
\]
The proof is now complete.
\end{proof}

The following result describes how $\F$ interacts with $\Gg$ and $\E$ through their Shilov boundaries.

\begin{cor}\label{cor_304}
	If $(x, a, p, s) \in \partial_s \F$, then $(x, a, p) \in b\E$ and $(s, -p) \in b\Gamma$.
\end{cor}

\begin{proof}
	Let $(x, a, p, s) \in \partial_s \F$. It follows from Theorem \ref{cor_bF} that $(x, a, p) \in b\E$ and $p \in \T$. By Corollary \ref{cor_309}, $(s, -p) \in \Gamma$ and thus by Theorem \ref{thm_bGam}, $(s, -p) \in b\Gamma$. 
\end{proof}

We mention here that the converse to the above corollary may not hold. For example, consider the quadruple $(x, a, p, s)=(i, 1, i, 1-i)$. By Theorems \ref{thm_bGam} and \ref{thm6.1}, $(x, a, p) \in b\E$ and $(s, -p) \in b\Gamma$. Let if possible, $(x, a, p, s) \in \partial_s \F$. Then by Theorem \ref{cor_bF}, there exists $w \in \C$ such that $|x|^2+|w|^2 \leq 1$ and $s=w-\overline{w}p$. This gives that $w=0$ and $s=0$, which is a contradiction. Thus, $(x, a, p, s) \notin \partial_s \F$. 

\medskip 

Having established the explicit description of $\partial_s \F$, we show that it is homeomorphic to a certain quotient space of $\T\times E_3$, where $E_3$ is the closed unit ball in $\R^3$.

\begin{cor}\label{cor_s1}
	The Shilov boundary $\pr_s \F$ is homeomorphic to the quotient space $(\T \times E_3)\slash \sim$, where $E_3$ is the closed unit ball in $\R^3$ and $\sim$ is the equivalence relation as in \eqref{eqn_equiv}. 
\end{cor}

\begin{proof}
	Let $E_3$ be the closed unit ball in $\R^3$, i.e., $E_3=\{(y_1, y_2, y_3) \in \R^3: y_1^2+y_2^2+y_3^2 \leq 1 \}$. It is clear from Theorem \ref{cor_bF} that the map given by
	\[
	g: \T \times E_3 \to \pr_s\F, \quad g(e^{i\theta}, x_2, x_3, x_4)=\left(e^{i\theta}(x_3+ix_4), \ -e^{i\theta}(x_3-ix_4),\ -e^{2i\theta}, \ 2e^{i\theta}x_2 \right)
	\]
	is a well-defined continuous surjective map. For $(e^{i\theta}, x_2, x_3, x_4), (e^{i\beta}, y_2, y_3, y_4) \in \T \times E_3$, it is not difficult to see that $g(e^{i\theta}, x_2, x_3, x_4)= g(e^{i\beta}, y_2, y_3, y_4)$ if and only if either $(e^{i\theta}, x_2, x_3, x_4) =(e^{i\beta}, y_2, y_3, y_4)$ or $(e^{i\theta}, x_2, x_3, x_4)=(-e^{i\beta}, -y_2, -y_3, -y_4)$. Let $\sim$ be the equivalence relation on $\T \times E_3$ defined by 
	\begin{align}\label{eqn_equiv}
		(e^{i\theta}, x_2, x_3, x_4)\sim (e^{i\beta}, y_2, y_3, y_4) \iff g(e^{i\theta}, x_2, x_3, x_4)= g(e^{i\beta}, y_2, y_3, y_4).
	\end{align}	
	In particular, the equivalence classes are precisely given by
	\[
	[(e^{i\theta}, x_2, x_3, x_4)]=\{(e^{i\theta}, x_2, x_3, x_4), -(e^{i\theta}, x_2, x_3, x_4)\}.
	\]
	Since $\T \times E_3$ and $\pr_s\F$ are compact Hausdorff spaces, it follows that the map defined as
	\[
	\tilde{g}: (\T \times E_3)\slash \sim \to \pr_s \F, \quad \tilde{g}([(e^{i\theta}, x_2, x_3, x_4)])=g(e^{i\theta}, x_2, x_3, x_4)
	\]
	is a homeomorphism from the quotient space $(\T \times E_3)\slash \sim$ onto the $\pr_s \F$.
\end{proof}

We now present one of the main results of this section.

\begin{thm}
	The hexablock $\He$ and the domain $\F$ are not biholomorphic to each other.
\end{thm}

\begin{proof}
	Let if possible, $h: \F \to \He$ be a biholomorphism. Since both $\mathbb{F}$ and $\mathbb{H}$ are both quasi-balanced bounded domains in $\mathbb{C}^4$ containing the origin, it follows from Remark 7 in \cite{Kosi} and Lemma \ref{extension} that $h$ extends to a biholomorphism from $\overline{\mathbb{F}}$ onto $\overline{\mathbb{H}}$. It follows from the discussion preceding Theorem \ref{thm_bF1} that $\mathbb{F}$ is relatively compact in $\mathbb{F}_{1\slash r}$ for every $r \in (0,1)$. Similarly, one can prove using Corollary 6.4 in \cite{Hexablock} that $\mathbb{H}$ is also relatively compact in $\mathbb{H}_{1\slash r}$ for every $r \in (0,1)$. By Remark 7 in \cite{Kosi}, the Minkowski functionals of $\F$ and $\He$ are continuous. It now follows from Part $(b)$ of Corollary 4 in \cite{Kosi} that $h(\partial_s \F)=\pr_s \He$ and so, $\partial_s \F$ and $\pr_s \He$ are homeomorphic to each other. By Corollary 8.22 in \cite{Hexablock}, $\pr_s \He$ is homeomorphic to $\T \times \pr \mathbb{B}_2$. Consequently, it follows from Corollary  \ref{cor_s1} that the quotient space $(\T \times E_3)\slash \sim$ is homeomorphic to $\T \times \pr \mathbb{B}_2$. Combining everything, we have  
	\begin{align}\label{eqn_404}
	(S^1 \times E_3)\slash \sim  \ \text{is homeomorphic to} \ S^1 \times S^3,
	\end{align}
	where $S^n=\{(x_1, \dotsc, x_{n+1}) \in \R^{n+1}: x_1^2+\dotsc+x_n^2=1\}$ and $\sim$ is the equivalence relation as in \eqref{eqn_equiv}. It is a routine exercise to verify that the two spaces in \eqref{eqn_404} are in fact are not homeomorphic. For the sake of completeness, we briefly outline one possible argument. To begin with, there is a natural action of $\mathbb{Z}_2$ on $S^1 \times E_3$ via the antipodal map $\alpha \mapsto -\alpha$ on $S^1 \times E_3$. In particular, this is an action on $S^1 \times E_3$ by $\mathbb{Z}_2=\{0, 1\}$, where $0$ acts as the identity map and $1$ acts as the antipodal map. Moreover, the action of $\mathbb{Z}_2$ on $S^1 \times E_3$ is free and properly discontinuous. Consequently, the quotient map $q: S^1 \times E_3 \to (S^1 \times E_3)\slash \sim$ is a covering map and $\mathbb{Z}_2$ is its group of covering transformations. Since $q$ is a covering map, it follows that $q$ is a surjective local homeomorphism. Also, $S^1 \times E_3$ is a $4$-manifold with boundary and so, $(S^1 \times E_3)\slash \sim$ is a $4$-manifold with boundary too. Since $S^1 \times S^3$ is a $4$-manifold without boundary, it cannot be homeomorphic to $(S^1 \times E_3)\slash \sim$, which contradicts \eqref{eqn_404}. Consequently, $\He$ is not biholomorphic to $\F$.
\end{proof}
	
We conclude this section by comparing the domains $\F$ and $\He$ through their distinguished boundaries. Theorem~8.21 of \cite{Hexablock} shows that 
$
b\He = \{(a, x_1, x_2, x_3) \in \overline{\He}_N : |x_3| = 1\}.
$
Since $\E$ embeds into $\He$ through the last three coordinates and into $\F$ through the first three, one might expect an analogous description for $\F$, namely
$
b\F = \{(x, a, p, s) \in \overline{\F} : |p| = 1\}.
$
However, the discussion after Corollary \ref{cor_304} shows that this is not true.

\section{Connection of the domain $\F$ with $\mu$-synthesis problem}\label{sec_004}	

\noindent Ghosh and Zwonek \cite{GhoshI} posed the problem of whether a $\mu$-synthesis approach can yield the domains $\mathbb{L}_n$, analogous to the known cases of the symmetrized bidisc $\Gg$ and the tetrablock $\E$. In this section, we address this question for $n=4$ by presenting a $\mu$-synthesis approach that produces the domain $\F$, which is biholomorphic to $\mathbb{L}_4$ as described in \eqref{eqn_L4_F}. Let us recall that given a linear subspace $E$ of $M_n(\mathbb{C})$, the
structured singular value is defined as
\[
\mu_E(A)=\left(\inf\{\|X\| \ : \ X \in E, \ \text{det}(I-AX)=0 \}\right)^{-1}   \quad (A
\in M_n(\mathbb{C})),
\] 
In case, there is no $X \in E$ with $\det(I-Ax)=0$, we define $\mu_E(A)=0$. The linear subspace $E$ is referred to as the structure. If $E=M_n(\mathbb{C})$, then $\mu_E=\|.\|$, whereas if $E$ is the space of scalar matrices, then
$\mu_E(A)=r(A)$, the spectral radius of $A$. For any linear subspace $E$ of
$M_n(\mathbb{C})$ containing the identity matrix $I$, we have that $r(A) \leq
\mu_E(A) \leq \|A\|$. 

\medskip 

The aim of the $\mu$-synthesis problem is to
find an analytic function $F$  from the open unit disc $\mathbb{D}$ to $M_n(\mathbb{C})$ subjected to finitely many
interpolation conditions such that $\mu_E(F(\lambda)) <1$ for every $\lambda \in \mathbb{D}$. The $\mu$-synthesis problem is the classical
Nevanlinna-Pick interpolation problem when $\mu_E=\|.\|$. When $\mu_E$ is the spectral radius, the problem of $\mu$-synthesis is referred to as the spectral Nevanlinna-Pick interpolation
problem. Attempts to solve various instances of the $\mu$-synthesis problem have led to the study of various domains such as the symmetrized bidisc $\Gg$, the tetrablock $\E$, the pentablock $\Pe$ and the hexablock $\He$, introduced in \cite{AglerIII}, \cite{Abouhajar}, \cite{AglerIV}, and \cite{Hexablock}, respectively.

\subsection*{The $\mu$-synthesis approach for $\Gg, \E, \Pe$ and $\He$} In order to explain how the $\mu$-synthesis approach gives rise to the domain $\F$, we first observe how the domains $\Gg$, $\mathbb{E}$, $\mathbb{P}$, and $\He$ arise as instances of $\mu$-synthesis. Let us denote
by $\mathbb{B}=\{A \in M_2(\C) : \|A\| <1  \}$ and $\mathbb{B}_{\mu_E}=\{A \in M_2(\mathbb{C}) \ : \ \mu_E(A) <1\}$ for a given linear subspace $E$ in $M_2(\mathbb{C})$. When $E $ comprises the $ 2
\times 2$ scalar matrices, $\mu_E$ is the spectral radius function. Agler-Young \cite{AglerYoung} introduced the symmetrized bidisc $\Gg$ and proved that 
\[
\Gg=\{(\text{tr}(A), \det(A)) \in \C^2 : r(A)<1 \}=\{(\text{tr}(A), \det(A)) \in \C^2 : \|A\|<1 \}.
\]
Thus, the images of $\mathbb{B}_{\mu_E}$ and $\mathbb{B}$ under the map $A \mapsto
(\text{tr}(A), \det(A))$ coincide.  When $E$ consists of $2 \times 2$ diagonal matrices,
$\mu_E$ is an intermediate cost function $\mu_{\text{tetra}}$. This case was studied
by Abouhajar, Young and White in \cite{Abouhajar} which led to the tetrablock $\E$. Indeed, it was proved in
\cite{Abouhajar} that 
\[
\E=\{(a_{11}, a_{22}, \det(A)) \in \C^3 : \mu_{\text{tetra}}(A)<1 \}=\{(a_{11},
a_{22}, \det(A)) \in \C^3 : \|A\|<1 \}.
\] 
In this case too, the images of $\mathbb{B}_{\mu_E}$ and $\mathbb{B}$ under the map
$A=[a_{ij}]_{i, j=1}^2 \mapsto (a_{11}, a_{22}, \det(A))$ are same. The next natural
step in this direction was taken by Agler, Lykova and Young in \cite{AglerIV}. The
authors considered the linear space $E$ consists of upper triangular matrices in $M_2(\C)$ with same diagonal entries. The $\mu$-function in this case is denoted by $\mu_{\text{penta}}$ and the corresponding $\mu$-synthesis problem gives rise to the pentablock $\Pe$. It was proved in \cite{AglerIV} that
\[
\Pe=\{(a_{21}, \text{tr}(A), \det(A)) \in \C^3 : \mu_{\text{penta}}(A)<1 \}=\{(a_{21},
tr(A), \det(A)) \in \C^3 : \|A\|<1 \}.
\]
It follows that $\mathbb{B}_{\mu_E}$ and $\mathbb{B}$ have same images under the map
$A=[a_{ij}]_{i, j=1}^2 \mapsto (a_{21}, \text{tr}(A), \det(A))$.
Much recently, Biswas, Pal and Tomar \cite{Hexablock} studied the case when $E$ is the subspace of upper triangular matrices in $M_2(\C)$, which is another natural candidate for a linear subspace. This choice of $E$ in $M_2(\C)$
generalizes all the previous linear subspaces for $\Gg, \mathbb{E}$ and
$\mathbb{P}$. The respective $\mu$-function denoted by $\mu_{\text{hexa}}$ satisfy
the following: 
\[
r(A) \leq \mu_{\text{tetra}}(A), \ \mu_{\text{penta}}(A) \leq \mu_{\text{hexa}}(A)
\leq \|A\|
\]
for every $A \in M_2(\C)$. Consider the mapping
\[
\pi_{hexa}: M_2(\mathbb{C}) \to \mathbb{C}^4, \quad A = \begin{pmatrix} a_{11} & a_{12} \\ a_{21} & a_{22} \end{pmatrix} \mapsto (a_{21}, a_{11}, a_{22}, \det (A)).
\]
Let $\mathbb{H}_\mu$ and $\He_N$ be the images of the norm-unit ball $\mathbb{B}$ and the $\mu_E$-unit ball $\mathbb{B}_{\mu_E}$ under the map $\pi_{hexa}$, respectively. Unlike the previously studied cases of $\mu$-synthesis approach, it turns out that $\He_N$ is not same as $\He_{\mu}$, and none of them are domains. The domain Hexablock $\He$ defined as in \eqref{eqn_He} turns out to be a natural choice in this endeavor as explained in \cite{Hexablock}. Moreover, we have 
\[
\overline{\He}=\overline{\He}_{\mu}=\widehat{\overline{\He}_N} \ (\text{the polynomial convex hull of $\overline{\He}_N$}).
\] 
Putting everything together, we have that the domains $\Gg, \E, \Pe$ and $\He$ arise from different cases of $\mu$-synthesis problem and the domains $\Gg, \E, \Pe$ are holomorphic retracts of the hexablock $\He$. An interested reader is referred to \cite{Hexablock} for further details. 

\medskip 

Motivated by this, a natural question arises whether the domain $\F$, introduced in \cite{GhoshI}, can be realized as a domain arising from a particular instance of $\mu$-synthesis. By definition, we have 
\begin{align*}
	\F=\{\pi_{\F}(A) : A \in M_2(\C), \|A\|<1\}, 
\end{align*}
where $\pi_{\F}: M_2(\C) \to \C^4$ defined as $\pi_{\F}(A)=(a_{11}, a_{22}, \det(A), a_{12}+a_{21})$ for $A=(a_{ij})_{i, j=1}^2 \in M_2(\C)$. To answer this question raised in \cite{GhoshI}, one needs to identify a linear subspace $E$ of $M_2(\C)$ such that $\F$ is the image of $\mu_E$-unit ball under the map $\pi_{\F}$, i.e., $\F=\pi_{\F}(\mathbb{B}_{\mu_E})$. From here onwards, set
\[
\F_{\mu_E}=\pi_{\F}(\mathbb{B}_{\mu_E})=\{\pi_{\F}(A) : A \in M_2(\C), \mu_E(A)<1\}
\]
for a linear subspace $E$ of $M_2(\C)$. If $E=M_2(\C)$, then $\mu_E=\|.\|$ and so, $\F=\pi_{\F}(\mathbb{B}_{\mu_E})$ holds trivially. Thus, a non-trivial description requires considering a proper linear subspace $E$ of $M_2(\C)$ such that $\F=\pi_{\F}(\mathbb{B}_{\mu_E})$. Our next result provides a characterization of all such linear subspaces in $M_2(\C)$.

\begin{thm}\label{thm_FvsF_mu}
	Let $E \subseteq M_2(\C)$ be a linear subspace and let $\F_{\mu_E}=\pi_{\F}(\mathbb{B}_{\mu_E})$. Then the following are equivalent:
	\begin{enumerate}
		\item $\F=\F_{\mu_E}$;
		\item $\mu_E(A)=\|A\|$ for all $A \in M_2(\C)$;
		\item for every pair of unit vectors $u, v \in \C^2$, there exists $A \in E$ with $\|A\|=1$ such that $Au=v$.
	\end{enumerate}
\end{thm}

\begin{proof}
The part $(2) \implies (1)$ holds trivially and the equivalence $(2) \iff (3)$ follows from Theorem 3.1 in \cite{rigid_SSV}. Thus, it remains to establish the part $(1) \implies (2)$. Suppose $\F=\F_{\mu_E}$. For any $A \in M_2(\C)$, we have that $\mu_E(A) \leq \|A\|$. Let if possible, there exist $A \in M_2(\C)$ such that $\mu_E(A)<\|A\|$. Choose $r>0$ such that $\mu_E(A)<r \leq \|A\|$. Define $B=(b_{ij})_{i, j =1}^2=r^{-1}A$. Since $\mu_E(r^{-1}A)=r^{-1}\mu_E(A)$, we have that $\mu_E(B)<1 \leq \|B\|$. Let $\pi_\F(B)=(x, a, p, s)$. Since $\mu_E(B)<1$, it follows that $(x, a, p, s) \in \F_{\mu_E}$ and so, $(x, a, p, s) \in \F$. Thus, there exists $C=(c_{ij})_{i, j=1}^2 \in M_2(\C)$ with $\|C\| <1$ such that $(x, a, p, s)=\pi_\F(C)$. Then
\[
(b_{11}, b_{22}, \det(B), b_{21}+b_{12})=(x, a, p, s)=(c_{11}, c_{22}, \det(C), c_{12}+c_{21})
\]
A simple calculation shows that $(b_{12}+b_{21}, b_{12}b_{21})=(c_{12}+c_{21}, c_{12}c_{21})=(s, ax-p)$. We have two possibilities: either $(b_{12}, b_{21})=(c_{12}, c_{21})$ or $(b_{12}, b_{21})=(c_{21}, c_{12})$. In other words, either $C=B$ or $C=B^t$. In either case, $\|C\|=\|B\|$ and so, $\|B\|<1$. This is a contradiction to the fact that $\|B\| \geq 1$. Therefore, $\mu_E(A)=\|A\|$ and the proof is now complete.
\end{proof}

Example 3.4 in \cite{rigid_SSV} shows that $\mu_{E_\theta}=\|.\|$ for the proper linear subspace of $M_2(\C)$ given by
\[
E_\theta=\left\{\begin{pmatrix}
	z_1 & w \\
	e^{i\theta} w & z_2
\end{pmatrix} : z_1, z_2, w \in \C \right\} \qquad (\theta \in \mathbb{R}).
\]
An immediate consequence of Theorem \ref{thm_FvsF_mu} combined with Example 3.4 in \cite{rigid_SSV} is the following result, which establishes the connection of the domain $\F$ with $\mu$-synthesis.

\begin{cor}\label{cor_F_theta}
For $\theta \in \R$ and  
$
E_\theta=\left\{\begin{pmatrix}
	z_1 & w \\
	e^{i\theta} w & z_2
\end{pmatrix} : z_1, z_2, w \in \C \right\}
$, we have $\F=\F_{\mu_{E_\theta}}$.
\end{cor}

As discussed earlier, the tetrablock $\E$ can be characterized as
\[
\E=\{(a_{11}, a_{22}, \det(A)) : \|A\|<1\}=\{(a_{11}, a_{22}, \det(A)) : \mu_\diag(A)<1\},
\] 
where $\mu_\diag$ is the structured singular value associated with $E_\diag$ in $M_2(\C)$. Since $\E \subseteq \F$, we now restrict ourselves to the linear subspace $E$ of $M_2(\C)$ such that $E_\diag \subseteq E$ and  $\F=\F_{\mu_E}$. In particular, one can take $E=E_\theta$ for any $\theta \in \R$, as noted in the last corollary. The next theorem shows that these are, in fact, the only possibilities for such a proper linear subspace $E$ of $M_2(\C)$.

\begin{thm}
	A proper linear subspace $E$ of $M_2(\C)$ with $E_{\diag} \subseteq E$ satisfies $\F_{\mu_E}=\F$ if and only if 
		\[
		E=\left\{\begin{pmatrix}
			z_1 & w \\
			e^{i\theta} w & z_2
		\end{pmatrix} : z_1, z_2, w \in \C \right\} \ \text{for some $\theta \in \R$.}
		\] 
\end{thm}

\begin{proof}
	We denote by $E_\theta=\left\{\begin{pmatrix}
		z_1 & w \\
		e^{i\theta} w & z_2
	\end{pmatrix} : z_1, z_2, w \in \C \right\}$ for $\theta \in \R$. Let $\F=\F_{\mu_E}$ for a proper subspace $E$ of $M_2(\C)$ such that $E$ contains $E_\diag$. It follows from Theorem \ref{thm_FvsF_mu} that $\mu_E(A)=\|A\|$ for all $A \in M_2(\C)$. For $E_{12}=\begin{pmatrix}
	0 & 1 \\
	0 & 0
	\end{pmatrix}$, it is not difficult to see that there does not exist $X \in E_{\diag}$ such that $\det(I-E_{12}X)=0$ and so, $\mu_{E_\text{diag}}(E_{12})=0$. Since $\|E_{12}\|=1$, we have that $E_\diag \subsetneq E$. Then $\dim(E)=3$ and so, we have
	\[
	E=\text{span}\left\{\begin{pmatrix}
		1 & 0\\
		0 & 0 \\
	\end{pmatrix}, \begin{pmatrix}
	0 & 0\\
	0 & 1 \\
	\end{pmatrix}, \begin{pmatrix}
	0 & \alpha\\
	\beta & 0 \\
	\end{pmatrix}\right\}
	\]
	for some $\alpha, \beta \in \C$. If $\beta=0$, then $E$ is the linear subspace of all upper triangular matrices in $M_2(\C)$ and the corresponding structured singular value, denoted by $\mu_\hexa$, coincides with $\|.\|$. This is not possible since there is no upper triangular matrix $X$ in $M_2(\C)$ such that $\det(I-E_{12}X)=0$ implying that $\mu_{\text{hexa}}(E_{12})=0$. Consider the case when $\alpha=0$ and define $E_{21}=\begin{pmatrix}
		0 & 0 \\
		1 & 0
	\end{pmatrix}$. In this case, $E$ is the space of all lower triangular matrices in $M_2(\C)$ and one can easily show that $\mu_E(E_{21})=0$, which is not possible as $\mu_E(E_{21})=\|E_{21}\|$ and $\|E_{21}\|=1$. Consequently, $\alpha, \beta \ne 0$ and so, we can re-write
	\[
	E=\text{span}\left\{\begin{pmatrix}
		1 & 0\\
		0 & 0 \\
	\end{pmatrix}, \begin{pmatrix}
		0 & 0\\
		0 & 1 \\
	\end{pmatrix}, \begin{pmatrix}
		0 & 1\\
		x & 0 \\
	\end{pmatrix}\right\}=\left\{\begin{pmatrix}
	z_1 & w \\
	xw & z_2
	\end{pmatrix}: z_1, z_2 ,w \in \C\right\}
	\]
for some $x \in \C\setminus \{0\}$. By Theorem \ref{thm_FvsF_mu}, there exists $A=\begin{pmatrix}
	z_1 & w \\
	xw & z_2
\end{pmatrix} \in E$ with $\|A\|=1$ such that $Ae_1=e_2$. Since $\|A\|=1$, we have $|w| \leq 1$. Also, the condition $Ae_1=e_2$ gives that $xw=1$ and so, $|x| \geq 1$. Again by Theorem \ref{thm_FvsF_mu}, there exists $B=\begin{pmatrix}
a & b \\
xb & c
\end{pmatrix}$ with $\|B\|=1$ such that 
\[
B\begin{pmatrix}
	0 \\
	x\slash |x|
\end{pmatrix}=e_1 \quad \text{and so,} \quad \begin{pmatrix}
bx\slash |x| \\
c\slash |x|
\end{pmatrix}=\begin{pmatrix}
1 \\ 0
\end{pmatrix}.
\]
Thus, $b=|x|\slash x$ and $c=0$. Note that $|b|=1$ and $\|B\|=1$ gives that $a=0$. Therefore, $B=\begin{pmatrix}
	0 & |x|\slash x \\
	|x| & 0
\end{pmatrix}$. Evidently $\|B\|=\max\{1, |x|\}=1$ and so, $|x|\leq 1$. Thus, $|x|=1$ and we can write $x=e^{i\theta}$ for some $\theta \in \R$. Hence, $E=E_\theta$. The converse follows directly from Corollary \ref{cor_F_theta}.  
\end{proof}

Note that the hypothesis  $E_{\diag} \subseteq E$ in the above theorem cannot be dropped. To see this, let 
	\[
	E=\left\{
	\begin{pmatrix}
		z & w_1\\
		w_2 & -z
	\end{pmatrix} : z, w_1, w_2 \in \C 
	\right\}.
	\]
Evidently, $E$ does not contain $E_\diag$ and $E \ne E_\theta$ for any $\theta \in \R$. We have by Example 3.5 in \cite{rigid_SSV} that $\mu_{E}=\|.\|$. It now follows from Theorem \ref{thm_FvsF_mu} that $\F=\F_{\mu_E}$.

\vspace{0.3cm}

\noindent \textbf{Funding.} The first named author was supported in part by Core Research Grant with Award No. CRG/2023/005223 from Anusandhan National Research Foundation (ANRF) of Govt. of India. The second named author was supported by the same grant of the first named author during the course of the paper.

\end{document}